\documentclass[11pt]{amsart}
\usepackage{amscd,amssymb,amsmath,latexsym,enumerate}
\usepackage[mathscr]{euscript}
\usepackage{epsfig}

\makeatletter
\newcommand{\specialcell}[1]{\ifmeasuring@#1\else\omit$\displaystyle#1$\ignorespaces\fi}
\makeatother

\newcommand{\hm}[1]{\textbf{*}\leavevmode{\marginpar{\tiny%
$\hbox to 0mm{\hspace*{-0.5mm}$\leftarrow$\hss}%
\vcenter{\vrule depth 0.1mm height 0.1mm width \the\marginparwidth}%
\hbox to 0mm{\hss$\rightarrow$\hspace*{-0.5mm}}$\\\relax\raggedright #1}}}

\usepackage{color}
\usepackage{tikz}

\newtheorem{theo}{Theorem}[section]
\newtheorem{proposi}[theo]{Proposition}
\newtheorem{lemma}[theo]{Lemma}
\newtheorem{coro}[theo]{Corollary}
\theoremstyle{definition}
\newtheorem{exam}[theo]{Example}
\newtheorem{rem}[theo]{Remark}
\newtheorem{defini}[theo]{Definition}

\newcommand{\Gg}{{\mathcal G}}
\newcommand{\Hh}{{\mathcal H}}

\newcommand{\Kk}{{\mathcal K}}
\newcommand{\Ll}{{\mathcal L}}

\newcommand{\Pp}{{\mathcal P}}

\newcommand{\Ss}{{\mathcal S}}

\newcommand{\1}{{\mathbf 1}}

\newcommand{\XX}{{\mathbf X}}

\newcommand{\CM}{{\mathbb C}}

\newcommand{\NM}{{\mathbb N}}

\newcommand{\RM}{{\mathbb R}}

\newcommand{\ZM}{{\mathbb Z}}

\newcommand{\sop}{\sigma^{\mathrm{op}}}               
\newcommand{\spec}{\mathrm{spec}}                     
\newcommand{\specess}{\mathrm{spec}_{\mathrm{ess}}}   
\DeclareMathOperator{\Orb}{Orb}                       
\newcommand{\PE}{{\rm \Psi E}}                        
\DeclareMathOperator{\prop}{prop}

\newcommand\pto{
   \unitlength0.1ex
   \begin{picture}(30,15)
   \put(15,16){\makebox(0,0)[]{\tiny $\Pp$}}
   \put(15,5){\makebox(0,0)[]{$\to$}}
   \end{picture}
}

\makeatletter
\newcommand\MyPairedDelimiter{%
  \@ifstar{\My@Paired@Delimiter{{}}}
          {\My@Paired@Delimiter{}}%
}
\newcommand\My@Paired@Delimiter[4]{%
  \newcommand#2{%
    \@ifstar{\start@PD{#1}{\delimitershortfall=-1sp}{#3}{#4}}
            {\start@PD{#1}{}{#3}{#4}}%
  }%
}
\newcommand\start@PD[5]{%
  #1\mathopen{\mathpalette\put@delim@helper{\put@delim{#2}{#3}{.}{#5}}}%
  #5%
  \mathclose{\mathpalette\put@delim@helper{\put@delim{#2}{.}{#4}{#5}}}%
}
\newcommand\put@delim@helper[2]{%
  \hbox{$\m@th\nulldelimiterspace=0pt #2#1$}%
}
\newcommand\put@delim[5]{%
  \setbox\z@\hbox{$\m@th#5{#4}$}%
  \setbox\tw@\null
  \ht\tw@\ht\z@ \dp\tw@\dp\z@
  #1#5%
  \left#2\box\tw@\right#3%
}

\makeatother
\MyPairedDelimiter*{\abs}{\lvert}{\rvert}
\MyPairedDelimiter*{\norm}{\lVert}{\rVert}
\MyPairedDelimiter{\set}{\{}{\}}

\DeclareMathOperator{\Plim}{\Pp\text{-}lim}

\begin{document}

\title[]{On the spectrum of operator families on discrete groups over minimal dynamical systems}
\author{Siegfried Beckus, Daniel Lenz, Marko Lindner, Christian Seifert}

\address{Mathematisches Institut\\
Friedrich-Schiller-Universit\"at, Jena\\
07743 Jena, Germany}
\email{siegfried.beckus@uni-jena.de}
\email{daniel.lenz@uni-jena.de}

\address{Technische Universit\"at Hamburg-Harburg\\
Institut f\"ur Mathematik\\
21073 Hamburg, Germany}
\email{marko.lindner@tuhh.de}

\address{Technische Universit\"at Hamburg-Harburg\\
Institut f\"ur Mathematik\\
21073 Hamburg, Germany and
Ludwig-Maxilians-Universit\"at M\"unchen\\
Mathematisches Institut\\
Theresienstra{\ss}e 39\\
80333 M\"unchen, Germany}
\email{christian.seifert@tuhh.de}

\begin{abstract}
It is well known that, given an equivariant and continuous (in a
suitable sense) family of selfadjoint operators in a Hilbert space
over a minimal dynamical system, the spectrum of all operators from
that family coincides. As shown recently similar results also hold
for  suitable families of non-selfadjoint operators in $\ell^p
(\ZM)$.  Here,  we generalize this  to a large class of bounded
linear operator families on Banach-space valued $\ell^p$-spaces over
countable discrete groups. We also provide equality of the
pseudospectra for operators in such a family.  A main tool for our
analysis are techniques from limit operator theory. 
\end{abstract}


\maketitle


\section{Introduction}
It is well-known that the spectrum of a selfadjoint operator $A$ in
a Hilbert space $\Hh$ can be approximated by the spectra of a
sequence $(A_n)$ of selfadjoint operators in $\Hh$ converging to $A$
in strong resolvent sense. This fact--sometimes called
semi-continuity of the spectrum--implies that the spectrum cannot
suddenly expand in the limit. Applied to a family of equivariant
selfadjoint operators over a minimal dynamical system this
semi-continuity yields that the spectrum (as a set) agrees for all
operators in that family, see \cite{Lenz2, Joh, CFKS, BIST,
LenzSeifertStollmann2014} for various versions of this theorem.
However, semi-continuity of the spectrum fails to be true for
non-selfadjoint operators \cite[Example IV.3.8]{Kato1980}.

Nevertheless, as shown in \cite{BeckusLenzLindnerSeifert2015}
constancy of the spectrum is still valid for suitable families of
operators in $\ell^p (\ZM)$ which are equivariant over dynamical
systems with a minimal $\ZM$ action, and \cite{Seifert2015} generalizes this to the case $\ell^p(\ZM^N)$.
This suggests  that minimality
of the dynamical system $(\Omega,\Gg,\alpha)$, where $\Omega$ is a
compact metric space, $\Gg$ is a countable group and $\alpha$ is
the group action of $\Gg$ on $\Omega$, implies constancy of the
spectrum of an equivariant family of (not nesessarily selfadjoint)
operators  $(A(\omega))_{\omega\in\Omega}$ on $\Gg$. Indeed,  as far
as dynamical systems are concerned the arguments given in
\cite{BeckusLenzLindnerSeifert2015, Seifert2015} are rather general (and simple)
and not restricted to  $\ZM$ and  it was already noted in the
introduction of \cite{BeckusLenzLindnerSeifert2015} that the results
can be extended to all situations where a suitable theory of limit
operators is at hand.

Here, we discuss how such a suitable theory of limit operator is
indeed at hand in a rather general situation dealing with Banach
space valued $\ell^p$ spaces. As a consequence we obtain  constancy
of the spectra  for $\Gg$ being countable discrete (not nesessarily abelian) and
$(A(\omega))_{\omega\in\Omega}$ is a (suitable) family of bounded
operators on $\ell^p(\Gg,Y)$, where $p\in[1,\infty]$ and $Y$ is
Banach space. In order to prove this fact we make use of the theory
of limit operators developed in
\cite{RabinovichRochSilbermann1998,RabinovichRochSilbermann2004,Lindner2006,CL08,CL10,RaRoRoe,LindnerSeidel2014,SpakWillett}.
The right topological framework turns out to be the so-called
$\Pp$-theory, see
\cite{RochSilbermann1989,ProssdorfSilbermann,RabinovichRochSilbermann2004,Seidel2014}.
Following  \cite{BeckusLenzLindnerSeifert2015} we introduce the
notion of pseudoergodic elements of $\Omega$, as first appeared in
\cite{Davies2001b} in a somewhat different way. We will show that
pseudoergodicity of elements is closely related to minimality of the
dynamical system. However, for pseudoergodic elements the set of
limit operators can be described rather easily, thus yielding the
assertion of our main theorem. We moreover prove constancy of the
pseudospectrum, which is an intensively studied object for
non-selfadjoint operators, see \cite{TrefethenEmbree2005} and
references therein.

In Section \ref{sec:P-theory} we collect basic notions and facts
from the $\Pp$-theory needed for our purpose. We will also introduce
the class of operators on groups we consider. Section
\ref{sec:dynamical_systems} deals with dynamical systems and shows
the interplay between pseudoergodicity of elements and minimality.
We will combine these two things in Section
\ref{sec:operator_over_DS} when we introduce the families of
operators over dynamical systems. There, we also state and prove our
main theorems. In Section~\ref{sec:operators_over_finitely_generated_groups} and Section~\ref{sec:operators_over_ZN} we specialize to the case of finitely
generated groups and to the group $\ZM^N$, which is mostly dealt
with in the literature
\cite{RabinovichRochSilbermann1998,RabinovichRochSilbermann2004,Lindner2006,CL10,LindnerSeidel2014}.

\section{Limit operators, the $\Pp$-theory and operators on discrete groups}
\label{sec:P-theory}


In this section we recall the $\Pp$-theory and the notion of limit operators, where we focus on Banach space valued sequence spaces $\XX$ on general countable discrete groups $(\Gg,+)$.

~

{\bf Fredholm operators and essential spectrum.}
We start with a complex Banach space $\XX$ and denote the sets of all bounded and all compact operators on $\XX$ by $\Ll(\XX)$ and $\Kk(\XX)$, respectively. Recall that $\Kk(\XX)$ forms a closed two-sided ideal in the Banach algebra $\Ll(\XX)$. An operator $A\in\Ll(\XX)$ is a {\em Fredholm operator} if its kernel has finite dimension and its image has finite co-dimension. By Calkin's theorem this is equivalent to invertibility of the coset $A+\Kk(\XX)$ in the quotient algebra $\Ll(\XX)/\Kk(\XX)$ -- also known as the Calkin algebra. For an operator $A\in \Ll(\XX)$ we write $\spec(A)$ for its spectrum and $\specess(A)$ for its {\em essential spectrum}, where the latter refers to the set of all $\lambda\in\CM$ such that $A-\lambda I$ is not Fredholm; it is the spectrum of the coset $A+\Kk(\XX)$ in the Calkin algebra.


~

{\bf Limit operators.}
An important (and in our setting the only) part of the spectrum of an operator is the essential spectrum, leading to the study of Fredholm operators, i.e., invertibility modulo compact operators. The first successful study for a large class of operators, the so-called band-dominated operators, on a space $\XX=\ell^p(\ZM^N,Y)$ with $\dim Y<\infty$ and $p\in (1,\infty)$, was done by Lange and Rabinovich \cite{LangeRabinovich1985} and finally with full proofs in \cite{RabinovichRochSilbermann1998} by Rabinovich, Roch and Silbermann, both are based on works of Muhamadiev \cite{Muh3} and Simonenko \cite{Simonenko68}. See Section 1.2 of \cite{CL10} for a more detailed history of the subject.

The argument in a nutshell: The Fredholm property of an operator $A$ on $\XX$ is invariant under compact perturbations such as arbitrary modifications to finitely many entries of the matrix $(a_{ij})_{i,j\in\ZM^N}$ that is associated to $A$; so the property must be ``encoded'' in this matrix ``at infinity''. To gather this information, one looks at limits of translates of the matrix (or, likewise, the operator $A$) as translation goes to infinity. Such limits are called limit operators of $A$, and there are usually many of them, each associated to another sequence of translations. \cite{RabinovichRochSilbermann1998} shows that $A$ is Fredholm iff all its limit operators are invertible (and their inverses are uniformly bounded -- this condition was later shown to be redundant \cite{LindnerSeidel2014}), so that the essential spectrum of $A$ is the union of spectra of its limit operators.

~

{\bf $\Pp$-theory.}
An important aspect is the topology in which limits of translates of $A$ are considered. This topology should be 
such that exactly the compact operators disappear (meaning that all their limit operators are zero). 

In the simple case $\XX=\ell^2(\ZM^N,\CM)$, compact operators are easily described: 
If we write $P_n$ for the operator on $\XX$ of multiplication by the characteristic function 
of the set $\{-n,...,n\}^N\subset\ZM^N$, then $K\in\Ll(\XX)$ is compact if and only if both
\begin{equation}\label{eq:comp}
\|K(I-P_n)\|\to 0\qquad\text{and}\qquad\|(I-P_n)K\|\to 0\qquad\text{as}\quad n\to\infty.
\end{equation}
On the matrix level, this means that $K=(k_{ij})$ has finite support or is the norm limit of a sequence
of such matrices.

The characterization of compact operators by \eqref{eq:comp} still holds if $\XX=\ell^p(\ZM^N,Y)$ with $p\in(1,\infty)$ and $\dim Y<\infty$ but
it is no longer true if $p\in\{1,\infty\}$ and/or $\dim Y=\infty$. (See e.g. Example~\ref{ex:RandSchr} below
for a natural instance of $\XX=\ell^2(\ZM^N,Y)$ with $\dim Y=\infty$.)

The $\Pp$-theory is aiming to unify all these cases by putting the sequence $\Pp:=(P_n)$ at center stage 
and going from there -- with tailor-made versions of compactness, Fredholmness and convergence. 
This is not limited to the particular sequence $\Pp$ introduced above and not even to $\ell^p$-spaces $\XX$:

~ 

Let $\XX$ be a complex Banach space. Following the lines of \cite{RabinovichRochSilbermann2004,SeidelSilbermann2012,Seidel2014}, we say that a bounded sequence $\Pp:=(P_n)_{n\in\NM}$ of operators in $\Ll(\XX)$ is called an \emph{approximate projection} whenever
\begin{itemize}
\item for all $n\in\NM$, the operator $P_n$ is not equal to the zero operator respectively the identity operator, and
\item for all $m\in\NM$, there exists an $N_m\in\NM$ such that $P_nP_m=P_mP_n=P_m$ for all $n\geq N_m$.
\end{itemize}
An approximate projection $\Pp$ is called an \emph{approximate identity} if for all $x\in\XX$ the inequality $\sup_{n\in\NM} \|P_nx\|\geq \|x\|$ holds.
This property provides that the elements of the family $\Pp$ asymptotically exploit the whole space $\XX$.
For an approximate projection $\Pp$, define $S_1:=P_1$ and $S_n:=P_{n+1}-P_n,\; n\geq 1$. Then $\Pp$ is called \emph{uniform} if the supremum
$\sup\{\|\sum_{k\in U}S_k\|\; :\; U\subseteq\NM \text{ finite} \}$ is finite.

\begin{defini}[$\Pp$-compact operators]
  Let $\Pp:=(P_n)_{n\in\NM}$ be an approximate projection. An operator $K\in \Ll(\XX)$ is called \emph{$\Pp$-compact} if \eqref{eq:comp} holds.
  Denote  by $\Kk(\XX,\Pp)$ the set of all $\Pp$-compact operators in $\Ll(\XX)$.
\end{defini}

One setback is that, unlike $\Kk(\XX)$, the new set $\Kk(\XX,\Pp)$ is not an ideal in the algebra $\Ll(\XX)$ of all bounded linear operators on $\XX$. Therefore pass to a smaller algebra -- the so-called multiplier algebra of $\Kk(\XX,\Pp)$ in $\Ll(\XX)$. This is the set
\[\Ll(\XX,\Pp)\ :=\ \{A\in \Ll(\XX):\; AK, KA\in \Kk(\XX,\Pp)\,\text{for all}\, K\in \Kk(\XX,\Pp)\}.\]
Then $\Ll(\XX,\Pp)$ is a closed, unital Banach subalgebra of $\Ll(\XX)$ containing $\Kk(\XX,\Pp)$ as a closed two-sided ideal, see \cite[Proposition 1.1.8]{RabinovichRochSilbermann2004}.
In Figure~\ref{fig:venn_diagramm} an illustration of the relations between the different algebras is sketched.

\begin{defini}[$\Pp$-Fredholm operators]
We say that $A\in \Ll(\XX,\Pp)$ is a \emph{$\Pp$-Fredholm operator} if $A$ is invertible in the quotient algebra $\Ll(\XX,\Pp)/\Kk(\XX,\Pp)$.
\end{defini}
By \cite[Corollary 12]{Seidel2014}, every Fredholm operator in $\Ll(\XX,\Pp)$ is a $\Pp$-Fredholm operator. The reverse implication requires $\dim Y$ to be finite.

\begin{figure}[tb]
  \setlength{\unitlength}{0.78mm}
  \begin{tabular}{|c|c|c|}
  \hline
  \rule{0mm}{2.5ex} & $\dim Y<\infty$ & $\dim Y=\infty$\\
  \hline
  \hline
  \begin{picture}(20,62)
    \put(10,31){\makebox(0,0)[]{$1<p<\infty$}}
  \end{picture} &

  \begin{picture}(60,62)
    \put(30,30){\oval(60,60)[]}
    \put(3,58){\makebox(0,0)[tl]{$\Ll(\XX)=\Ll(\XX,\Pp)$}}
    \put(38,22){\oval(40,40)[]}
    \put(38,40){\makebox(0,0)[t]{$\Kk(\XX)=\Kk(\XX,\Pp)$}}
    \put(9.5,22){\circle{0.5}}
    \put(10,22.5){\makebox(0,0)[bl]{\scriptsize $I$}}
  \end{picture} &

  \begin{picture}(60,62)
    \put(30,30){\oval(60,60)[]}
    \put(3,58){\makebox(0,0)[tl]{$\Ll(\XX)$}}
    \put(34,26){\oval(48,48)[]}
    \put(13,48){\makebox(0,0)[tl]{$\Ll(\XX,\Pp)$}}
    \put(38,22){\oval(36,36)[]}
    \put(23,38){\makebox(0,0)[tl]{$\Kk(\XX,\Pp)$}}
    \put(42,18){\oval(24,24)[]}
    \put(33,28){\makebox(0,0)[tl]{$\Kk(\XX)$}}
    \put(14.5,22){\circle{0.5}}
    \put(15,22.5){\makebox(0,0)[bl]{\scriptsize $I$}}
  \end{picture} \\

  \hline
  \begin{picture}(20,62)
    \put(10,31){\makebox(0,0)[]{$p=1,\infty$}}
  \end{picture} &

  \begin{picture}(60,62)
    \put(30,30){\oval(60,60)[]}
    \put(3,58){\makebox(0,0)[tl]{$\Ll(\XX)$}}
    \put(26,33){\oval(40,40)[]}
    \put(9,51){\makebox(0,0)[tl]{$\Ll(\XX,\Pp)$}}
    \put(40,20){\oval(35,35)[]}
    \put(54,4){\makebox(0,0)[br]{$\Kk(\XX)$}}
    \put(34,25){\makebox(0,0)[]{$\Kk(\XX,\Pp)$}}
    \put(14.5,30){\circle{0.5}}
    \put(15,30.5){\makebox(0,0)[bl]{\scriptsize $I$}}
  \end{picture} &

  \begin{picture}(60,62)
    \put(30,30){\oval(60,60)[]}
    \put(3,58){\makebox(0,0)[tl]{$\Ll(\XX)$}}
    \put(26,33){\oval(40,40)[]}
    \put(9,51){\makebox(0,0)[tl]{$\Ll(\XX,\Pp)$}}
    \put(42,18){\oval(31,31)[]}
    \put(54,4){\makebox(0,0)[br]{$\Kk(\XX)$}}
    \put(31,28){\oval(30,30)[]}
    \put(19,41){\makebox(0,0)[tl]{$\Kk(\XX,\Pp)$}}
    \put(11,30){\circle{0.5}}
    \put(11.5,30.5){\makebox(0,0)[bl]{\scriptsize $I$}}
  \end{picture} \\
  \hline
  \end{tabular}
\caption{The spaces $\Ll(\XX)$, $\Ll(\XX,\Pp)$, $\Kk(\XX,\Pp)$ and $\Kk(\XX)$ (the compact operators on $\XX$). See\ \cite{Lindner2006,Seidel2014}}
  \label{fig:venn_diagramm}
\end{figure}

Compact operators are exactly the operators which, when multiplied from the right, turn strong 
convergence into norm-convergence. This property is mimicked now in the $\Pp$-setting:
\begin{defini}[$\Pp$-strong convergence]
  Let $\Pp:=(P_n)_{n\in\NM}$ be an approximate identity. A sequence $(A_n)$ in $\Ll(\XX)$ \emph{converges $\Pp$-strongly} to $A\in \Ll(\XX)$ if
  \[\lim\limits_{n\to\infty}\norm{K(A_n-A)} + \norm{(A_n-A)K} = 0\]
  for all $K\in\Kk(\XX,\Pp)$. In this case, we write $A = \Plim A_n$ or $A_n\pto A$.
\end{defini}

$\Ll(\XX,\Pp)$ is the natural playground for the study of $\Pp$-strong convergence. For example, $\Pp$-strong convergence of a sequence $(A_n)$ from $\Ll(\XX,\Pp)$ implies that the sequence is bounded, the limit is unique and again in $\Ll(\XX,\Pp)$, see \cite[Section 1.1.4]{RabinovichRochSilbermann2004}.

For a sequence $(A_n)_{n\in\NM}$ in $\Ll(\XX,\Pp)$ we have $A = \Plim A_n$ if and only if $(A_n)$ is bounded and $$\lim\limits_{n\to\infty}\norm{P_m(A-A_n)} + \norm{(A-A_n)P_m}= 0$$
for all $m\in\NM$, see \cite[Proposition 1.1.4]{RabinovichRochSilbermann2004}.

\begin{rem}
\label{rem:P-topology}
  Let $\Pp:=(P_n)_{n\in\NM}$ be an approximate identity. For $n\in\NM$ define the seminorm $p_n: \Ll(\XX,\Pp)\to \RM$ by
  $$p_n(A) := \norm{P_n A} + \norm{AP_n},  \qquad A\in \Ll(\XX,\Pp).$$
  Then, on bounded subsets of $\Ll(\XX,\Pp)$, the family $(p_n)_{n\in\NM}$ of seminorms defines a metriziable topology, for which the $\Pp$-limit is the limit with respect to this metric.
\end{rem}

\begin{rem}
Let $\XX=\ell^p(\ZM^N,Y)$ with $1<p<\infty$ and $\dim Y <\infty$, and let $\Pp=(P_n)$, where $P_n$ is
the operator of multiplication by the characteristic function of $\{-n,...,n\}^N\subset\ZM^N$.
Then $\Ll(\XX,\Pp) = \Ll(\XX)$, $\Pp$-strong convergence is equivalent to $*$-strong convergence (that is strong convergence $A_n\to A$ and $A_n^*\to A^*$), and $\Pp$-Fredholmness coincides with Fredholmness.
\end{rem}

Historically, $\Pp$-theory has its roots in Simonenko \cite{Simonenko68}, started in the present form with 
Pr\"ossdorf, Roch and Silbermann \cite{ProssdorfSilbermann,RochSilbermann1989} and recently culminated
in Seidels work -- see \cite{Seidel2014} for his excellent survey.


~

{\bf Operators on discrete groups.}
Now we pass to $\ell^p$-spaces on a discrete group $\Gg$ rather than fixing $\Gg=\ZM^N$.
So let $p\in[1,\infty]$, $Y$ a Banach space, $(\Gg,+)$ a discrete group (not nessesarily abelian or finitely generated), and let $\XX:=\ell^p(\Gg,Y)$.

For $n\in\NM$ let $\Gg_n\subseteq \Gg$ be such that $\varnothing\neq \Gg_n \neq \Gg$, and such that for all $m\in\NM$ there exists $N_m$ satisfying $\Gg_m\subseteq \Gg_n$ for all $n\geq N_m$. Furthermore, suppose $\bigcup_{n\in\NM} \Gg_n = \Gg$. For $x\in\XX$ define
$$P_n(x_g)_{g\in \Gg} := (\1_{\Gg_n}(g) x_g)_{g\in \Gg}$$
where $\1_{\Gg_n}$ denotes the characteristic function of the set $\Gg_n$. Then $\Pp:=(P_n)_{n\in\NM}$ defines an approximate identity which can be checked by inspection.
Furthermore, $\Pp$ is even uniform.
In the following we always suppose that $\Pp$ is defined by such an exhaustion defining a uniform approximate identity.

For $g\in \Gg$ and $x\in \XX$ let $V_gx:=(x_{h+g})_{h\in \Gg}$ denote the \emph{shift by $g$}. For a sequence $(g_n)_{n\in\NM}$ in $\Gg$ we write $g_n \to \infty$ if it leaves every finite set, i.e.,\ if for every finite set $F\subseteq \Gg$ there exists an $n_F\in\NM$ with $g_n\notin F$ for all $n\geq n_F$.
With this notation at hand we define limit operators: 

\begin{defini}[limit operators]
Let $A\in \Ll(\XX,\Pp)$ and $g:=(g_n)_{n\in\NM}$ in $\Gg$ with $g_n\to \infty$. Then
\[A_g\ :=\ \Plim_{n\to\infty} V_{g_n}AV_{-g_n},\]
if it exists, is called the \emph{limit operator of $A$ with respect to $g$}.
The set $\sop(A)$ of all limit operators of $A$ is called the \emph{operator spectrum of $A$}. An operator $A\in \Ll(\XX,\Pp)$ is called \emph{self-similar} if $A\in \sop(A)$.
\end{defini}

The importance of the operator spectrum becomes clear in the study of Fredholm and $\Pp$-Fredholm operators. For so-called rich band-dominated operators on $\XX$, which is a proper subalgebra of $\Ll(\XX,\Pp)$, we have, by \cite{Seidel2014,LindnerSeidel2014}, that
\[
A \textrm{ is Fredholm}\quad\Longrightarrow\quad
A \textrm{ is $\Pp-$Fredholm}\quad\iff\quad
\textrm{ all } B\in\sop(A) \textrm{ are invertible.}
\]
Both ``$\Longrightarrow$'' arrows already hold in $\Ll(\XX,\Pp)$, so that we have the following theorem:

\begin{theo}[see {\cite[Theorem 1.28]{SeidelSilbermann2012}}]
\label{thm:spectral_inclusion_groups} Let
$A \in \Ll(\XX,\Pp)$. Then
$$
\specess(A)\ \supseteq\ \underset{B\in\sop(A)}{\bigcup}\spec(B)
$$
holds. In particular,
$$
\spec(A)=\specess(A)
$$
holds if $A$ is self-similar.
\end{theo}

\section{Dynamical systems over discrete groups}
\label{sec:dynamical_systems}
\bigskip

Let $(\Gg,+)$ be a discrete group with countably many elements,
$\Omega$ a compact metric space, such that there exists a
representation $\alpha$ of $\Gg$ in $\{T:\Omega\to\Omega:\;
T\,\text{homeomorphism}\}$, i.e., $\alpha\colon \Gg\to \{T:\Omega\to\Omega:\;
T\,\text{homeomorphism}\}$ such that
\[\alpha(g+h) = \alpha(g)\circ\alpha(h) \quad(g,h\in \Gg).\]
Then we call $(\Omega,\Gg,\alpha)$ a \emph{dynamical system}.
Note that if $\Gg$ is abelian, then 
the homeomorphisms in the image of $\alpha$ commute.

\begin{defini}[limit sets]
Let $(\Omega,\Gg, \alpha)$ be a dynamical system.
For $\omega\in\Omega$ we write
\[L(\omega):=\{\nu\in\Omega:\; \exists\, (g_n)_{n\in\NM} \text{ in  } \Gg \text{ s.t. }  g_n\to \infty \text{ and } \alpha(g_n)(\omega)\to \nu\},\]
the \emph{limit set of $\omega$}.
\end{defini}

\begin{lemma}\label{lem:properties-L-omega}
Let $(\Omega,\Gg, \alpha)$ be a dynamical system, $\omega\in\Omega$. Then $L(\omega)$ is compact, non-empty and
$\alpha$-invariant.
\end{lemma}

\begin{proof} For a finite set $F\subseteq \Gg$ we define
$$L_F (\omega):=\overline{\{ \alpha
(g) (\omega) : g\notin F\}}.$$
Then $L_F (\omega)$ is closed, and
hence, compact. Moreover, as $\Gg$ is infinite every finite
intersection of sets of the form $L_F (\omega)$, $F\subseteq \Gg$, is
non-empty. Thus, by compactness, the infinite intersection
$$L(\omega) =\bigcap_{F \subseteq \Gg, F \,\text{finite}} L_F (\omega)$$
of compact sets is compact and non-empty. Moreover, for each $g\in \Gg$ the equation
$$\bigcap_{F \subseteq \Gg, F \,\text{finite}} L_F (\omega) = \bigcap_{F \subseteq \Gg, F \,\text{finite}} L_{F+g}
(\omega)$$ holds and so the $\alpha$-invariance of $L(\omega)$ follows.
\end{proof}

\begin{defini}[Pseudoergodic elements] Let $(\Omega,\Gg, \alpha)$ be a dynamical system.
Then  $\omega\in\Omega$  is called \textit{pseudoergodic} if
$L(\omega) = \Omega$ holds. The set of pseudoergodic elements is denoted by $\Omega_\PE$.
\end{defini}

\begin{defini}[Minimality] Let $(\Omega,\Gg, \alpha)$ be a dynamical system.
For $\omega\in\Omega$ we define the \emph{orbit} of $\omega$ by
\[\Orb(\omega):=\{\alpha(g)(\omega):\; g\in\Gg\}.\]
We say that $(\Omega,\Gg,\alpha)$ is \emph{minimal} if $\Orb(\omega)$ is dense in $\Omega$ for every $\omega\in\Omega$.
\end{defini}

Clearly, pseudoergodic elements have dense orbits in $\Omega$. We next show the converse under an additional assumption.

\begin{lemma}
\label{lem:pseudo-ergodic}
Let $(\Omega,\Gg, \alpha)$ be a dynamical system, $\omega\in\Omega$ not isolated and $\Orb(\omega)$ dense in $\Omega$.
Then $\omega$ is pseudoergodic.
\end{lemma}

\begin{proof}
As $\Orb(\omega)$ is dense and $L(\omega)$ is closed and $\alpha$-invariant it suffices to show that it contains
$\omega$. Since $\omega$ is not isolated, there exists a sequence $(\omega_n)$ in $\Omega$
converging to $\omega$ such that the $\omega_n$ are pairwise
different and none of them equals $\omega$. Let $d$ denote a metric on $\Omega$ generating the topology.
By the denseness of the orbit $\Orb(\omega)$, for each $n\in\NM$ there exists $g_n\in \Gg$ with $ d(\alpha(g_n)(\omega), \omega_n) \leq \frac{1}{3} d(\omega, \omega_n)$. The assumption
on the $(\omega_n)$ provides that the set $\{g_n : n\in \NM\}$ is
infinite. Thus, $(g_n)_{n\in\NM}$ leaves every finite set and the equations
$$\lim_{n\to \infty} \alpha(g_{n})(\omega) = \lim_{n\to
\infty} \omega_{n}  = \omega$$
hold, proving $\omega\in L(\omega)$.
\end{proof}

\begin{proposi}
\label{prop:minimal_all_elements_pseudo}
  Let $(\Omega,\Gg,\alpha)$ be a dynamical system.
  The following assertions are equivalent:
  \begin{itemize}
   \item[$(i)$] $(\Omega,\Gg,\alpha)$ is minimal.
   \item[$(ii)$] $\Omega = \Omega_{\PE}$.
   \item[$(iii)$] $\Omega_{\PE}$ is closed and non-empty.
  \end{itemize}
\end{proposi}

\begin{proof}
  To prove the implication $(i)\Rightarrow (ii)$ we choose $\omega \in \Omega$ and show $L(\omega) = \Omega$.
  By Lemma~\ref{lem:properties-L-omega}, $L(\omega)$ is non-empty, closed and
  $\alpha$-invariant. Hence, it contains at least the closure of one orbit.
  Thus, by minimality it equals $\Omega$. The implication $(ii)\Rightarrow (iii)$ is clear. Finally, $(iii)\Rightarrow (i)$ can be proven as follows. Let $\omega\in\Omega_\PE$. Since $\Omega_\PE$ is $\alpha$-invariant we observe $\Orb(\omega)\subseteq \Omega_\PE$. Since $\Orb(\omega)$ is dense and $\Omega_\PE$ is closed we obtain
  $\Omega = \Omega_\PE$, so every element is pseudoergodic and therefore has a dense orbit.
\end{proof}

\section{Operators over dynamical systems}
\label{sec:operator_over_DS}
\bigskip

In this section we define operators parametrized by dynamical systems and study their spectral theory.

\begin{defini}
\label{def:family_of_operators}
Let $(\Omega,\Gg,\alpha)$ be a dynamical system.
Let $A:\Omega\to \Ll(\XX,\Pp)$ such that
\begin{itemize}
\item[$(i)$] $A(\alpha(g)(\omega))=V_{g}A(\omega)V_{-g}$ for all $\omega\in\Omega$, $g\in \Gg$.\hfill\textit{(Equivariance)}
\item[$(ii)$] The map $\omega\mapsto A(\omega)$ is $\Pp$-continuous, i.e.,\ $\omega_n\to \omega$ implies \hfill\textit{(Continuity)}
$$\Plim A(\omega_n)= A(\omega).$$
\end{itemize}
Then $A$ is called a \emph{family of operators over $(\Omega,\Gg,\alpha)$}.
\end{defini}

Let $(\Omega,\Gg,\alpha)$ be a dynamical system and $A:\Omega\to \Ll(\XX,\Pp)$ be a family of operators over $(\Omega,\Gg,\alpha)$.
If $A(\Omega)$ is bounded then the continuity of $A(\cdot)$ as in $(ii)$ is exactly the continuity with respect to the topology as in Remark \ref{rem:P-topology}.

\begin{proposi}
\label{prop:operator_spectra}
Let $(\Omega,\Gg,\alpha)$ be a dynamical system,
$A:\Omega\to \Ll(\XX,\Pp)$ a family of
operators over $(\Omega,\Gg,\alpha)$ and $\omega\in\Omega$. Then
$$
\sop(A(\omega))=\left\{ A(\nu)\;:\; \nu\in L(\omega) \right\}
$$
for all $\omega\in\Omega$.
\end{proposi}

\begin{proof}
We first show ``$\subseteq$''. Let $\omega\in \Omega$ and $B\in\sop(A(\omega))$.
Then there exists a sequence $(g_n)_{n\in\NM}$ in $\Gg$
with $g_n\to\infty$, such that
$$
B=\Plim
V_{g_n} A(\omega) V_{-g_n} =\Plim A(\alpha(g_n)(\omega)).
$$
Since $g_n\to\infty$ and $\Omega$ is compact we can select a subsequence
$(g_{n_k})_{k\in\NM}$ such that $(\alpha(g_{n_k})(\omega))_{k\in\NM}$ is
convergent to $\nu\in \Omega$.
Thus, $\nu \in L(\omega)$. Using
the $\Pp$-continuity of $A(\cdot)$ we obtain
$B = A(\nu)$.

\medskip

We now turn to prove the converse inclusion ``$\supseteq$''. For $\omega\in\Omega$ and $\nu \in L(\omega)$,
there is a sequence $(g_n)_{n\in\NM}$ such that $g_n\to\infty$ and $\lim_{n\to\infty} \alpha(g_n)(\omega)=\nu$.
Then the $\Pp$-continuity of $A(\cdot)$
implies
$$
A(\nu)=\Plim
A(\alpha(g_n)(\omega))= \Plim V_{g_n}A(\omega)V_{-g_n},
$$
leading to $A(\nu)\in\sop(A(\omega))$.
\end{proof}

\begin{coro}
\label{cor:operator_spectrum_pseudoergodic}
Let $(\Omega,\Gg, \alpha)$ be a dynamical system and $A:\Omega\to \Ll(\XX,\Pp)$ a family of operators over $(\Omega,\Gg,\alpha)$. Then
$$
\sop(A(\omega)) = \left\{ A(\nu)\;:\; \nu\in \Omega \right\}
$$
for all $\omega\in \Omega_\PE$. In particular, $A(\omega)$ is self-similar for each $\omega\in \Omega_\PE$.
\end{coro}

\begin{proof}
Since $\omega$ is pseudoergodic we have $L(\omega) = \Omega$. Thus, the assertion follows from the previous Proposition~\ref{prop:operator_spectra}.
\end{proof}

\begin{theo}
\label{thm:constant_spectra_pseudo-ergodic_groups}
Let $(\Omega,\Gg, \alpha)$ be a dynamical system and $A:\Omega\to \Ll(\XX,\Pp)$ a family of operators over $(\Omega,\Gg,\alpha)$.
Set
$$
\Sigma:=\underset{\omega\in \Omega}{\bigcup}\spec(A(\omega)).
$$
Then
$$
\spec(A(\omega))=\specess(A(\omega))=\Sigma
$$
for all $\omega\in \Omega_{\PE}$.
\end{theo}

\begin{proof}
  For pseudoergodic $\omega\in\Omega$ we have, taking into account Theorem \ref{thm:spectral_inclusion_groups} and Corollary \ref{cor:operator_spectrum_pseudoergodic},
  \[\Sigma\supseteq \spec(A(\omega)) \supseteq \specess(A(\omega))\supseteq\underset{B\in\sop(A)}{\bigcup}\spec(B)  \supseteq\underset{\nu\in\Omega}{\bigcup}\spec(A(\nu)) = \Sigma. \qedhere\]
\end{proof}

\begin{coro}
\label{cor:constant_spectra_minimal_groups}
Let $(\Omega,\Gg,\alpha)$ be a minimal dynamical system and $A:\Omega\to \Ll(\XX,\Pp)$ a family of operators over $(\Omega,\Gg,\alpha)$.
Set
$$
\Sigma:=\underset{\omega\in\Omega}{\bigcup}\spec(A(\omega)).
$$
Then
$$
\spec(A(\omega))=\specess(A(\omega))=\Sigma
$$
for all $\omega\in \Omega$.
\end{coro}

\begin{proof}
According to Proposition~\ref{prop:minimal_all_elements_pseudo} minimality implies $\Omega_{\PE} = \Omega$. Thus, the statement
follows directly from the previous Theorem~\ref{thm:constant_spectra_pseudo-ergodic_groups}.
\end{proof}

As we have seen, the spectrum of a family of operators over a minimal dynamical system is constant.
Next, we show that in this setup also the norms of the resolvents are equal for all operators of the family.
While this is an immediate consequence of spectral constancy in case of selfadjoint operators, in the non-selfadjoint case this might be a bit surprising.
We will need a little preparation.

\begin{lemma}
\label{lem:limit_operator_norm}
  Let $A\in \Ll(\XX,\Pp)$, $B\in \sop(A)$. Then $\|B\|\leq \|A\|$.
\end{lemma}

\begin{proof}
  This is a direct consequence of \cite[Proposition 1.1.17]{RabinovichRochSilbermann2004} and the fact that $\|P_n\|\leq 1$ for all $n\in\NM$.
\end{proof}

\begin{theo}
\label{thm:resovlent_norms}
  Let $(\Omega,\Gg,\alpha)$ be a dynamical system, $A:\Omega\to \Ll(\XX,\Pp)$ a family of operators over $(\Omega,\Gg,\alpha)$.
  Let $\Sigma:=\underset{\omega\in\Omega}{\bigcup}\spec(A(\omega))$,
  $\lambda\in\CM\setminus \Sigma$, and $\omega,\nu\in\Omega_\PE$. Then
  \[\|(\lambda-A(\omega))^{-1}\| = \|(\lambda-A(\nu))^{-1}\|.\]
  If $(\Omega,\Gg,\alpha)$ is minimal then
  \[\|(\lambda-A(\omega))^{-1}\| = \|(\lambda-A(\nu))^{-1}\|\]
  holds for all $\omega,\nu\in\Omega$.
\end{theo}

\begin{proof}
  Clearly, $\sop(\lambda-A(\omega)) = \lambda-\sop(A(\omega))$. Therefore, without loss of generality, we may suppose that $\lambda=0$.
  Since $A(\omega)^{-1}$ trivially is a regularizer for $A(\omega)$, i.e.,\ $A(\omega)^{-1}A(\omega)-I, A(\omega)A(\omega)^{-1}-I\in \Kk(\XX,\Pp)$,
  by \cite[Theorem 16]{Seidel2014} (which was proven there only for $\Gg=\ZM^N$, 
  but generalizes to arbitrary groups without any difficulty) we have
  \[\sop(A(\omega)^{-1}) = \{B^{-1}:\; B\in\sop(A(\omega))\} = \{A(\nu)^{-1}:\; \nu\in\Omega\}\]
  due to Corollary \ref{cor:operator_spectrum_pseudoergodic}.
  Thus, by Lemma \ref{lem:limit_operator_norm} we obtain
  \[\|A(\nu)^{-1}\| \leq  \|A(\omega)^{-1}\|.\]
  Interchanging the roles of $\omega$ and $\nu$ yields the first part of our assertion.

  In case of minimality, we just note that $\Omega_\PE = \Omega$ holds by Proposition \ref{prop:minimal_all_elements_pseudo}.
\end{proof}

The super-level sets of the resolvent norm yield the so-called pseudospectra:

\begin{defini}[pseudospectrum]
  Let $A\in \Ll(\XX)$, $\varepsilon>0$. Then we define the \emph{$\varepsilon$-pseudospectrum of $A$} by
  \[\spec_\varepsilon(A):=\{\lambda\in\CM:\; \|(\lambda-A)^{-1}\|> \tfrac{1}{\varepsilon}\},\]
  where we set $\|(\lambda-A)^{-1}\|:=\infty$ if $\lambda-A$ is not boundedly invertible.
\end{defini}

For normal (in particular selfadjoint) operators $A$, $\spec_\varepsilon(A)$ is the $\varepsilon$-neighbourhood of the spectrum.
In general, it is larger. Part of the attraction of pseudospectra is that $\spec_\varepsilon(A)$ equals the union of
$\spec(A+T)$ over all perturbations $\|T\|<\varepsilon$. The interplay between resolvents of $A$ and semigroup $(e^{tA})_{t\geq 0}$ moreover
allows bounds on the transient growth of the semigroup in terms of the pseudospectrum (whereas the spectrum just determines
the asymptotics as $t\to\infty$). See the monograph \cite{TrefethenEmbree2005} for more aspects of pseudospectra.

In view of Theorem \ref{thm:resovlent_norms} we can also obtain equality of pseudospectra for families of operators.

\begin{coro}
\label{cor:pseudospectra}
  Let $(\Omega,\Gg,\alpha)$ be a dynamical system, $A:\Omega\to \Ll(\XX,\Pp)$ a family of operators over $(\Omega,\Gg,\alpha)$, $\varepsilon>0$.
  Then
  \[\spec_\varepsilon(A(\omega)) = \spec_\varepsilon(A(\nu))\]
  holds for all $\omega,\nu\in\Omega_\PE$.
  If $(\Omega,\Gg,\alpha)$ is minimal then
  \[\spec_\varepsilon(A(\omega)) = \spec_\varepsilon(A(\nu))\]
  holds for all $\omega,\nu\in\Omega$.
\end{coro}

Note that Corollary \ref{cor:pseudospectra} contains our corresponding theorem on spectra as a special case (after letting $\varepsilon\to 0$).

\begin{exam}[Operators of finite propagation]
  Let $A\in \Ll(\XX)$. Then we say that $A$ is \emph{of finite propagation} if 
  \[\prop(A):=\sup\{d(g,h):\; g,h\in\Gg,\, (A\delta_g)(h) = 0\}<\infty,\]
  where $d$ is the path distance in $\Gg$ and 
  \[\delta_g(h):=\begin{cases} 1 & h = g,\\ 0 & \text{else}\end{cases} \quad(g\in\Gg).\]
  Then, $\prop(A)$ is called \emph{propagation} of $A$. Note that each operator of finite propagation belongs to $\Ll(\XX,\Pp)$. 

  Now, let $A:\Omega\to \Ll(\XX,\Pp)$ be a family of operators over $(\Omega,\Gg,\alpha)$ with finite propagation.
  Typical models of this kind appear in the study of Schr\"odinger operators with random potentials on $\Gg$.
\end{exam}

\begin{exam}[Band-dominated operators]
  Since $\Ll(\XX,\Pp)$ is closed in $\Ll(\XX)$, norm-limits of operators of finite propagation are contained in $\Ll(\XX,\Pp)$.
  Such operators are called \emph{band-dominated operators}, and are typically studied for the case $\Gg=\ZM^N$, see e.g.~\cite{RabinovichRochSilbermann1998,RabinovichRochSilbermann2004,Lindner2006,CL10,Seidel2014}, 
  and on $\ell^2(\Gg)$ for a finitely generated group $\Gg$ in \cite{Roe2005} by means of $C^*$-algebra techniques.  
\end{exam}

\section{Operators over finitely generated groups}
\label{sec:operators_over_finitely_generated_groups}

\bigskip

Let us now assume that $(\Gg,+)$ is finitely generated, i.e.,\ there exists $\Ss\subseteq \Gg$ finite such that for every $g\in\Gg$ there exist $n\in\NM$ and $s_1,\ldots,s_n\in\Ss\cup(-\Ss)$ such that $g = s_1+\ldots+s_n$.

Let $\Omega$ be a compact metric space and for $s\in\Ss$ let $\alpha(s):\Omega\to\Omega$ be a homeomorphism. We may assume that
$\alpha(-s) = \alpha(s)^{-1}$ for all $s\in \Ss$ (either it is a definition of $\alpha(-s)$ or it is an assumption, depending on whether $-s\in\Ss$ or not), and that
if $g = s_1+s_2+\ldots s_n\in\Gg$ then
\[\alpha(g):=\alpha(s_1)\circ \alpha(s_2)\circ\ldots\circ \alpha(s_n)\]
is well-defined, i.e.,\ if $g = s_1+\ldots+s_n = s_1'+\ldots +s_m'$ then
\[\alpha(s_1)\circ\ldots\circ \alpha(s_n) = \alpha(s_1')\circ\ldots\circ \alpha(s_m').\]

\begin{lemma}
  For $g,h\in\Gg$ we have
  $\alpha(g+h) = \alpha(g)\circ\alpha(h)$. Hence, $(\Omega,\Gg,\alpha)$ is a dynamical system.
\end{lemma}

\begin{proof}
  Write $g = s_1+\ldots+s_n$ and $h = s_1'+\ldots+s_m'$ with $s_1,\ldots,s_n,s_1',\ldots,s_m'\in\Ss\cup(-\Ss)$. Then $g+h = s_1+\ldots+s_n + s_1'+\ldots+s_m'$ and so the equalities
  \begin{align*}
    \alpha(g+h) & = \alpha(s_1)\circ\ldots \circ\alpha(s_n)\circ\alpha(s_1')\circ\ldots\circ\alpha(s_m') = \alpha(g)\circ\alpha(h)
  \end{align*}
  hold.
\end{proof}

Recall that for $g\in\Gg$ we have defined the shift $V_gx:=(x_{h+g})_{h\in\Gg}$ $(x\in \XX)$.
Consider $A:\Omega\to \Ll(\XX,\Pp)$ such that
\begin{itemize}
\item[$(i')$] $A(\alpha(s)(\omega))=V_{s}A(\omega)V_{-s}$ for all $\omega\in\Omega$, $s\in \Ss$.\hfill\textit{(Equivariance)}
\item[$(ii)$] The map $\omega\mapsto A(\omega)$ is $\Pp$-continuous.\hfill\textit{(Continuity)}
\end{itemize}

\begin{lemma}
\label{lem:suff_cond_family_operator}
  Let $A:\Omega\to \Ll(\XX,\Pp)$ satisfy condition $(i')$. Then, for $g\in\Gg$, $\omega\in\Omega$, we have
  \[A(\alpha(g)(\omega)) = V_{g} A(\omega)V_{-g},\]
  i.e.,\ $A$ fulfills $(i)$ in Definition \ref{def:family_of_operators}. If in addition $A$ also satisfies $(ii)$, then, $A$ is a family of operators over $(\Omega,\Gg,\alpha)$.
\end{lemma}

\begin{proof}
  Write $g = s_1+\ldots+s_n$ with $s_1,\ldots,s_n\in\Ss\cup(-\Ss)$.
  This yields
  \begin{align*}
    A(\alpha(g)(\omega)) & = A(\alpha(s_1)\circ \alpha(s_2)\circ\ldots\circ \alpha(s_n)(\omega)) = V_{s_1} A(\alpha(s_2)\circ\ldots\circ \alpha(s_n)(\omega))V_{-s_1} \\
    & = \ldots = V_{s_1}\ldots V_{s_n} A(\omega) V_{-s_n} \ldots V_{-s_1} = V_{g} A(\omega) V_{-g}. \qedhere
  \end{align*}
\end{proof}

Thus, we can apply Corollary \ref{cor:constant_spectra_minimal_groups} to obtain constancy of the spectrum whenever $(\Omega,\Gg,\alpha)$ is minimal.

\begin{exam}[Operators on the discrete Heisenberg group]
  We now give an example of a non-abelian group, where our results can be applied as well. Let $\Gg:=H_3(\ZM)$ be the discrete Heisenberg group, i.e.\ the group of $3\times 3$-matrices of the form
  \[\begin{pmatrix}
      1 & a & c \\
      0 & 1 & b \\
      0 & 0 & 1
    \end{pmatrix}\]
  with $a,b,c\in\ZM$, equipped with matrix multiplication. The group is non-abelian and generated by the two elements described by $(a,b,c)\in \set{(1,0,0),(0,1,0)}$.
  Typical examples of operators on the discrete Heisenberg group are Schr\"odinger operators, which are operators of finite propagation.
  They can be used to model random walks on the corresponding Cayley graph of the group, see for example \cite{BeguinValetteZuk1997}.
\end{exam}

\section{Operators over $\ZM^N$}
\label{sec:operators_over_ZN}
\bigskip

A particular example of a finitely generated group $\Gg$ is $\ZM^N$, with generators $e_j:=(\delta_{jk})_{k\in\{1,\ldots,N\}}$ ($j\in\{1,\ldots,N\}$), where $\delta$ denotes the Kronecker delta.
Let $\Omega$ be a compact metric space. For $j\in\{1,\ldots,N\}$ let $\alpha(e_j):\Omega\to\Omega$ be a homeomorphism. Assume that $\alpha(e_j)\circ\alpha(e_k) = \alpha(e_k) \circ \alpha(e_j)$ holds for all $j,k\in\{1,\ldots,N\}$.
We now extend $\alpha$ to $\ZM^N$: for $g\in\ZM^N$ define
\[\alpha(g):=\alpha(e_1)^{g_1}\circ \ldots\circ \alpha(e_n)^{g_n}.\]
The fact that $\alpha(e_j)$ commutes with $\alpha(e_k)$ for all $j,k\in \{1,\ldots,n\}$ yields that $\alpha(g)$ is well-defined. By definition, $\alpha$ is a representation of $\ZM^N$ in the set of homeomorphisms on $\Omega$.
Thus, $(\Omega,\ZM^N,\alpha)$ is a dynamical system.

Consider $A:\Omega\to \Ll(\XX,\Pp)$ such that
\begin{itemize}
\item[$(i')$] $A(\alpha(e_j)(\omega))=V_{e_j}A(\omega)V_{-e_j}$ for all $\omega\in\Omega$, $j\in\{1,\ldots,N\}$.\hfill\textit{(Equivariance)}
\item[$(ii)$] The map $\omega\mapsto A(\omega)$ is $\Pp$-continuous.\hfill\textit{(Continuity)}
\end{itemize}
According to Lemma~\ref{lem:suff_cond_family_operator}, the map $A:\Omega\to \Ll(\XX,\Pp)$ defines a family of operators over $(\Omega,\ZM^N,\alpha)$.
Thus, if $(\Omega,\ZM^N,\alpha)$ is minimal, then, Corollary \ref{cor:constant_spectra_minimal_groups} yields constancy of the spectrum.

\medskip

Let us remark that much more is known for operator families over $\ZM^N$.
Let $S^{N-1}:=\{\eta \in \RM^N:\; \abs{\eta} = 1\}$ be the unit sphere in $\RM^N$. For $R>0$, $\eta\in S^{N-1}$ and a neighborhood $U\subseteq S^{N-1}$ of $\eta$ we call
\[W_{R,U} := \{k\in \ZM^N:\; \abs{k}>R,\, \tfrac{k}{\abs{k}}\in U\}\]
a \emph{neighborhood at infinity of $\eta$}.
Let $g = (g_n)_{n\in\NM}$ be a sequence in $\ZM^N$ such that $g_n\to \infty$. We say that \emph{$g$ tends to infinity in the direction of $\eta\in S^{N-1}$} if for every neighborhood at infinity $W_{R,U}$ of $\eta$, there exists $n_0\in\NM$ such that
\[h_n\in W_{R,U} \quad(n\geq n_0).\]
Let $\eta\in S^{N-1}$ and $A\in \Ll(\XX,\Pp)$. The \emph{local operator spectrum} $\sop_\eta(A)$ of $A$ at $\eta$ is the set of all limit operators $A_g$ of $A$ with respect to sequences $g$ tending to infinity in the direction of $\eta$.

\begin{proposi}[{\cite[Proposition 2.3.2]{RabinovichRochSilbermann2004}}]
  Let $A\in \Ll(\XX,\Pp)$. Then $\sop(A) = \bigcup_{\eta\in S^{N-1}} \sop_\eta(A)$.
\end{proposi}

For $\eta\in S^{N-1}$ and $\omega\in\Omega$ define the \emph{limit set of $\omega$ at direction $\eta$} by
\[L^\eta(\omega):= \{ \nu\in\Omega :\; \text{$\exists\,(g_n)_{n\in\NM}$ in $\ZM^N$ tending to infinity in direction $\eta$ s.t. $\alpha(g_n)(\omega) \to \nu$}\}.\]

\begin{lemma}
  Let $\omega\in\Omega$. Then
  \[L(\omega) = \bigcup_{\eta\in S^{N-1}} L^\eta(\omega).\]
\end{lemma}

\begin{proof}
  We only have to show ``$\subseteq$, the other inclusion is trivial.
  Let $\nu\in L(\omega)$. Then there exists $(g_n)_{n\in\NM}$ in $\ZM^N$ such that $g_n\to\infty$ and $\alpha(g_n)(\omega)\to \nu$.
  Without loss of generality we may assume $g_n\neq 0$ for all $n\in\NM$. Due to compactness of $S^{N-1}$, the sequence $(\tfrac{g_n}{|g_n|})$ has a convergent subsequence. Let $\eta$ be its limit.
  This implies $\nu\in L^\eta(\omega)$.
\end{proof}

\begin{proposi}
\label{prop:limit_sets} Let $\omega\in\Omega$ and $\eta\in S^{N-1}$. Then $L^\eta(\omega)$ is
non-empty, compact and invariant under $\alpha$.
\end{proposi}

\begin{proof}
Let $(g_n)$ in $\ZM^N$ tend to infinity in direction $\eta$.
By compactness of $\Omega$, the sequence $(\alpha(g_n)(\omega))_{n\in\NM}$ has a convergent subsequence.
Thus, the set $L^\eta(\omega)$ is non-empty.

\medskip

Let
$\nu=\lim_{n\to\infty}\alpha(g_n)(\omega)\in L^\eta(\omega)$, where $(g_n)_{n\in\NM}$ tends to infinity in direction $\eta$.
Due to continuity, for $h\in\ZM^N$ we have
$$\lim_{n\to\infty}\alpha(h+g_n)(\omega) = \alpha(h)\left(\lim_{n\to\infty} \alpha(g_n)(\omega)\right) = \alpha(h)(\nu).$$
Clearly, also $(h+g_n)_{n}$ tends to infinity in direction $\eta$.
Hence, $\alpha(h)(\nu)\in L^\eta(\omega)$ follows.

\medskip

We now show closedness of $L^\eta(\omega)$, implying that it is compact.
Let
$(\nu_k)_{k\in\NM}$ in $L^\eta(\omega)$, $\nu_k\to \nu$. For
$k\in\NM$ there exists $(g^k_n)_{n\in\NM}$ tending to infinity in direction $\eta$ such that
$\nu_k=\lim_{n\to\infty} \alpha(g^k_n)(\omega)$.
Thus, for a suitable subsequence $(n_k)_{k\in\NM}$ we obtain that $(g^k_{n_k})_{k\in\NM}$ tends to infinity in direction $\eta$ and
$\lim_{k\to\infty}\alpha(g^k_{n_k})(\omega) = \nu$. Consequently,
$L^\eta(\omega)$ is closed.
\end{proof}

Note that in case $N=1$, we have $S^0 = \{-1,1\}$. Then $L^{+1}(\nu)$ is the so-called $\omega$-limit set and $L^{-1}(\nu)$ the $\alpha$-limit-set of $\nu$.

\begin{exam}[Random Schr\"odinger operators in $L_2(\RM^N)$] \label{ex:RandSchr}
  Let $Q:=(0,1)^N$ be the unit hypercube.
  Let $0\leq v\in L_\infty(\RM^N)$, $v|_{\RM^N\setminus Q} = 0$.
  Let $l>0$ and equip $[0,l]^{\ZM^N}$ with the product topology.
  Note that this topology is induced by the metric
  \[d(\omega,\nu):= \sum_{k\in\ZM^N} 2^{-|k|} \min\{|\omega_k-\nu_k|,1\}.\]
  Let $\Omega\subseteq [0,l]^{\ZM^N}$ be closed and translation invariant.
  Then $\Omega$ is compact and $\ZM^N$ naturally acts on $\Omega$. Let $\alpha$ be this action.
  For $\omega\in\Omega$, define
  \[V_\omega:=\sum_{k\in \ZM^N} \omega_k v(\cdot-k) \in L_\infty(\RM^N),\]
  and
  $H(\omega)$ in $L_2(\RM^N)$ by
  \begin{align*}
    D(H(\omega)) & := W_2^2(\RM^N),\\
    H(\omega) u & := -\Delta u + V_\omega u.
  \end{align*}
  Then $H(\omega) \geq 0$ holds for all $\omega\in\Omega$.

  Define $A(\omega):= (H(\omega)+1)^{-1}$. By \cite[Proposition 1.1.8]{RabinovichRochSilbermann2004}, a Combes-Thomas estimate \cite[Theorem 2.4.1]{Stollmann2001}
  and identifying $L_2(\RM^N) = \ell^2(\ZM^N,L_2(Q))=:\XX$ we obtain $A(\omega)\in \Ll(\XX,\Pp)$ for all $\omega\in\Omega$ where $P_n:=\1_{\{-n,\ldots,n\}^N}$ for $n\in\NM$.
  Furthermore, $\norm{A(\omega)}\leq 1$ holds for all $\omega\in\Omega$.
  By construction, $(i)$ is satisfied.
  Let us turn to condition $(ii)$, so let $(\omega^n)$ be a sequence in $\Omega$, $\omega^n\to\omega$.
  By the second resolvent identity it suffices to check that $\Plim_{n\to\infty}V_{\omega^n}=V_\omega$ as multiplication operators. Since $(V_{\omega^n})$ is uniformly bounded, we have to verify that
$$\lim\limits_{n\to\infty}\|P_m(V_{\omega}-V_{\omega^n})\| + \|(V_{\omega}-V_{\omega^n})P_m\|= 0 \quad (m\in\NM) .$$
  Let $m\in\NM$ and $\varepsilon>0$. Since $\omega^n \to\omega$, there exists $n_0\in \NM$ such that $|\omega^n_k-\omega_k|\leq \varepsilon$ for all $n\geq n_0$ and $k\in\{-m,\ldots,m\}^N$.
  Hence, $\|P_m(V_{\omega}-V_{\omega^n})\| + \|(V_{\omega}-V_{\omega^n})P_m\|$ tends to zero if $n$ goes to infinity.

  Thus, if $(\Omega,\ZM^N,\alpha)$ is minimal, then, the spectrum of $(A(\omega))_{\omega\in\Omega}$ is constant.
  Note that, for $\lambda\neq 0$, we have $\lambda\in \spec(A(\omega))$ if and only if $\frac{1-\lambda}{\lambda}\in \spec(H(\omega))$ (recall that $0\in \spec(A(\omega))$ since $H(\omega)+1$ is not bounded).
  Hence, we can also conclude constancy of the spectrum for the unbounded operator family $(H(\omega))_{\omega\in\Omega}$ with the help of our techniques.
\end{exam}

\end{document}